\documentclass[letterpaper, 10pt, conference]{ieeeconf} 
\IEEEoverridecommandlockouts
\let\labelindent\relax

\usepackage{graphics} 
\usepackage{epsfig} 
\usepackage{mathptmx} 
\usepackage{times} 
\usepackage{amsmath} 
\usepackage{amssymb} 
\usepackage[utf8]{inputenc} 
\usepackage[T1]{fontenc}
\usepackage{url}
\usepackage{ifthen}
\usepackage{cite}
\usepackage{amsthm}
\usepackage{amssymb,amsfonts,graphicx,nicefrac,mathtools, verbatim}
\usepackage{enumitem}
\usepackage{breqn}
\usepackage{listings}
\usepackage{color}

\newtheorem{theorem}{Theorem}

\newtheorem{lemma}{Lemma}
\newtheorem*{remark}{Remark}
\newtheorem{assumption}{Assumption}

\newcommand{\T}{\mathcal{T}}
\newcommand{\abs}[1]{| #1 |}

\title{\LARGE \bf
Dynamic Median Consensus Over Random Networks
}

\author{Shuhua Yu, Yuan Chen, Soummya Kar
\thanks{S. Yu and S. Kar are with Department of Electrical and Computer Engineering,
        Carnegie Mellon University, 15213 Pittsburgh, The PA
        {\tt\small \{shuhuay, soummyak\}@andrew.cmu.edu}}
\thanks{The work of S. Yu and S. Kar was partially supported by the National Science Foundation under grant CNS-1837607.}
}

\begin{document}

\maketitle 
\thispagestyle{empty}
\pagestyle{empty}

\begin{abstract}
This paper studies the problem of finding the median of $N$ distinct numbers distributed across networked agents. Each agent updates its estimate for the median from noisy local observations of one of the $N$ numbers and information from neighbors. We consider a undirected random network that is connected on average, and a noisy observation sequence that has finite variance and almost surely decaying bias. We present a \textit{consensus+innovations} algorithm with \textit{clipped innovations}. Under some regularity assumptions on the network and observation model, we show that each agent's local estimate converges to the set of median(s) almost surely at an asymptotic sublinear rate. Numerical experiments demonstrate the effectiveness of the presented algorithm.
\end{abstract}

\section{Introduction}
The past decades have seen increasing interests in decentralized control and coordination of large scale networked systems. A canonical problem in decentralized control is consensus. The objective of the consensus problem is to ensure that networked agents reach agreement on a common decision. This paper focuses on dynamic \textit{median} consensus, where the local observations are dynamic and noise corrupted, and the considered network is random and connected on average.   

The problem of consensus over multi-agent networks has a rich literature. In particular, the problem of average (or mean) consensus, i.e., finding the mean of the initial states of network agents has been investigated in \cite{kempe2003gossip,olfati2004consensus,xiao2004fast}. The convergence of average consensus algorithms have been studied in switching topology and networks  with time-delays \cite{olfati2004consensus}, networks with random link failures \cite{kar2007distributed}; and \cite{xiao2004fast, kar2008topology} study topology and weight matrix designs for fast convergence respectively. However, average consensus protocols can be vulnerable to attacks in large scale networks like IoT \cite{leblanc2012consensus}, e.g., a single attack on the initial state of one agent can arbitrarily manipulate the network average. 

As a result, consensus on a more robust statistical measure, like median, has been of research interest \cite{ben2015robust, franceschelli2016finite, pilloni2016robust, dashti2019dynamic}. Median consensus also finds applications in multi-robot systems \cite{vasiljevic2020dynamic}. Specifically, \cite{franceschelli2014finite, franceschelli2016finite} propose a continuous time protocol that finds the median value of networked agents, whereas \cite{pilloni2016robust} studies median consensus problem in the presence of matching perturbations to the agents' dynamics. The paper \cite{dashti2019dynamic} studies dynamic median consensus where agents have locally time-varying signals and the network is open in the sense that agents may come and leave during the protocol execution.

In this paper, we consider a local observation model in which the network agents obtain noisy measurements of $N$ real numbers (local values or parameters) with decaying bias (expected observation error) and white noise. This observation model arises in scenarios where local measurement perturbations are believed to be decaying sublinearly in expectation while the white noises from measuring devices or environment perturbations are unavoidable. This model also subsumes static or white noise corrupted observations as special cases. The inter-agent communication network we consider is undirected, random, independently and identically distributed (i.i.d.) and only required to be \emph{connected on average}.

Our contributions are as follows. We present an algorithm, under the above observation and network model, that converges to the set of median(s) almost surely and characterize the asymptotic sublinear convergence rate. Moreover, we tackle the observation noise by a carefully designed recursive averaging scheme, and develop a technical lemma of independent interest to show diminishing local averaging errors by carefully balancing the effects of decaying bias and white noise.

To the best of our knowledge, existing works that are most similar to this paper are references~\cite{pilloni2016robust} and~\cite{dashti2019dynamic}. The authors of~\cite{pilloni2016robust} study median consensus in static networks. They consider a setup where each agent perfectly observes its local value, and the agents collaborate to compute the median of all local values. The agents' interactions are subject to disturbances that are \textit{deterministically} bounded in magnitude. Reference~\cite{dashti2019dynamic} proposes an algorithm to track the \textit{instantaneous} median of a set of local reference signals over time-varying graphs. Agents perfectly observe their local reference signals, the reference signals have determinstically bounded time-derivatives, and, the graph must be connected at (almost) all times. 

In contrast to~\cite{pilloni2016robust} and~\cite{dashti2019dynamic}, in this paper, we consider the scenario where agents are unable to perfectly observe their local values. Each agent instead makes a measurement of its local value, subject to measurement noise, which, unlike~\cite{pilloni2016robust} and~\cite{dashti2019dynamic}, is \textit{not} determinstically bounded in magnitude. Further, unlike~\cite{dashti2019dynamic}, we present a median consensus algorithm that does \textit{not} require the network's graph topology to be connected at all times. 

The rest of the paper is organized as follows. In section II, we formulate the median consensus problem with assumptions. In section III we present our algorithm and the main result. Section IV is concerned with intermediate lemmas and is concluded with the proof of the main result. In section IV we provide some numerical experiments on networks with varying connectivity.

\textit{Notation:} We specify the communication network of agents. Agents may exchange messages over a time-varying network denoted as $G(t) = (V, E(t))$ at discrete time $t$. The set of agents $V$ is fixed and $\abs{V} = N$, and we use $[N]$ to denote all indices in the sequel. The set of communication links $E(t)$ is time-varying. For each agent $n$, $\Omega_n(t)$ denotes the set of neighbors at time $t$. We denote the Laplacian of network $G(t)$ as $L(t) = D(t) - A(t)$ where $D(t)$ is a diagonal matrix whose $i$-th diagonal entry is the number of neighbors of agent $i$ and $A_{i, j}(t) = 1$ if there exists a link between node $i$ and node $j$ at time $t$, otherwise $A_{i, j}(t) = 0$. A network is connected if and only if the second largest eigenvalue of its Laplacian matrix is positive \cite{mohar1991laplacian}. We use $\|\cdot\|_2$ to denote Euclidean norm for vectors. The term $\mathbf{1}_N$ represents an $N$-dimensional column vector whose elements are all $1$.

\section{Problem formulation}

We consider the problem that a set of $N$ agents aim to find the median(s) of a set of $N$ distinct numbers $\{\theta_1, \theta_2, \ldots, \theta_N\}$ through local time-sequential observations and and neighboring information exchange over a random network. Each agent $n$ can only access noisy observations of $\theta_n$. We consider a dynamic observation model for each agent $n$,
\begin{equation}
\label{eq:obsmodel}
	\theta_n(t) = \theta_n + w_n(t) + \nu_n(t)
\end{equation}
where $\{w_n(t)\}_{n \in [N], t \ge 0}$ is an i.i.d. noise sequence with zero mean and finite variance, and the $\nu_n(t)$ is decaying almost surely in that $\mathbb{P}\left(\abs{\nu_n(t)} \le v_0(t+1)^{-\delta}\right) = 1$ for some positive constants $\delta$ and $v_0$.

This formulation subsumes several models of local observations. Generally, this model considers observation errors with decaying bias and white noise. The decaying bias may be a result of local computation, or when the local observations are a result of another underlying iterative and convergent computational process, whereas, white noise is due to unavoidable measurement errors from physical devices, environment, etc. Moreover, by taking $\delta$ to be arbitrarily large we recover the special model corresponding to unbiased local observations, whereas, by setting $w_n(t)$ to be 0 almost surely the problem is reduced to the simplest static case. 

For convenience of notation, let us assume $\theta_1 < \theta_2 < \ldots < \theta_N$, and define $d_\min$ as the minimal gap between distinct elements in $\{\theta_n\}_{n = 1, \ldots, N}$,
\begin{equation}
\label{eq:dmin}
	d_{\min} = \min\{\abs{\theta_i - \theta_j}: i \ne j, i, j \in [N]\} > 0.
\end{equation}
We define the set of medians $\Theta$
\begin{equation*}
\Theta = 
\begin{cases}
	\{\theta_{\frac{N+1}{2}}\}, & N \text{ is odd}, \\
	\left[\theta_{\frac{N}{2}}, \theta_{\frac{N}{2}+1}\right], & N \text{ is even}, 
\end{cases}
\end{equation*}
and the distance function $\text{dist}(x, \Theta) := \min_{\theta \in \Theta}|x - \theta|$.

We now formally state our assumptions on the inter-agent communication network and the observation model.
\begin{assumption}
\label{ass:network}
	The time-varying inter-agent communication network is given by a random graph sequence $\{G(t)\}$. For all $t$, $G(t)$ is an undirected graph and the Laplacian sequence $\{L(t)\}$ associated with $\{G(t)\}$ is an i.i.d. sequence whose expectation, denoted as $\bar{L}$, exists and satisfies $\lambda_2(\bar{L}) > 0$.
\end{assumption}
\begin{remark}
	The assumption clearly subsumes the typical case of a static network $G = (V, E)$ where $G$ is connected. The above assumption is quite general and subsumes phenomena such as random link failures, i.e., in which the links in $G$ have  failure probabilities in $[0, 1)$, which models many practical networks such as wireless networks. Again, this assumption contrasts this work with \cite{dashti2019dynamic} in that this assumption allows models such as gossip based protocols in which no network instance (the stochastic realization) is connected, while \cite{dashti2019dynamic} requires each network instance to be connected.
\end{remark}

\begin{assumption}
\label{ass:noise}
	The observation noise $w_n(t)$ is i.i.d. distributed over time and independent across agents. For any $n$ and $t$, $\mathbb{E}\left(w_n(t)\right) = 0$ and $\mathrm{Var}(w_n(t)) < \infty$. For any $n$, the perturbation $v_n(t)$ decays a.s. in that $\mathbb{P}\left(\abs{v_n(t)} \le v_0(t+1)^{-\delta}\right) = 1$ for some positive constants $\delta$ and $v_0$.
\end{assumption}

Let $(\Omega, \mathcal{F})$ be the probability space where random variables $L(t), w_n(t)$ are defined, and let $\{\mathcal{F}(t)\}$ be the corresponding natural filtration, i.e., $\mathcal{F}(t)$ is the sigma algebra $\sigma\left(\{L(t')\}_{t' = 0, \ldots, t}, \{w_n(t)\}_{n \in [N]}\right)$. In this paper, unless otherwise stated, all inequalities involving random variables hold almost surely (a.s.).

\section{Algorithm and main result}
We present the following algorithm to estimate $\Theta$. Let $x_n(t)$ be the estimate of agent $n$ at time $t$, we update its estimate as follows, 
\begin{equation}
\begin{split}
\label{al:each}
	x_n(t+1) 
	= x_n(t) & - \beta_t \sum_{m \in \Omega_n(t)}^{}(x_n(t) - x_m(t))) \\ 
	& - \alpha_t k_n(t)(x_n(t) - \bar{\theta}_n(t)), 
\end{split}
\end{equation}
where $\bar{\theta}_n(t)$ is some recursive weighted average of observations $\{\theta_n(t')\}_{0 \le t' \le t}$ given by, for some $0 < c_\mu \le 1, 0 < \mu < 1$, 
\begin{equation}
\label{eq:localavg}
	\bar{\theta}_n(t+1) = (1 - \frac{c_\mu}{(t+1)^\mu})\bar{\theta}_n(t) + \frac{c_\mu}{(t+1)^\mu} \theta_n(t), 
\end{equation}
and the step sizes 
\begin{equation}
\begin{split}
\label{a1:coef}
	\alpha_t = \frac{\alpha_0}{(t + 1)^{\tau_1}}, \ \beta_t = \frac{\beta_0}{(t+1)^{\tau_2}},
\end{split}
\end{equation}
satisfy $\alpha_0, \beta_0 > 0, 0 < \tau_2 < \tau_1 < 1$. The term $k_n(t)$ in \eqref{al:each}  is the clipping operator, defined as,
\begin{equation}
\label{al:k}
	k_n(t) = 
	\begin{cases}
	1, & \text{if } |x_n(t) - \bar{\theta}_n(t)| \le \gamma_t, \\
	\gamma_t |x_n(t) - \bar{\theta}_n(t)|^{-1}, & \text{otherwise,}
	\end{cases}
\end{equation}
for 
\begin{equation}
\label{al:gamma}
	\gamma_t = \frac{\gamma_0}{(t+1)^{\tau_3}},
\end{equation}
with $\gamma_0 > 0$ and
\begin{equation}
\label{al:t3}
	\tau_3 < \min\{1 - \tau_1, 0.5 \delta_0\},
\end{equation} 
where $\delta_0 = 1 - \bar{\epsilon}$ and $\mu = \delta_0$ for any $0 < \bar{\epsilon} < 1$ if $\delta \ge 1$, otherwise $\delta_0 = \delta$ and $\delta \le \mu < 1$. 

We next refer to the algorithm given by \eqref{al:each}-\eqref{al:t3} as DMED for short, a Distributed Median Estimator for Dynamic observations, which is a \textit{consensus+innovations} type estimator \cite{kar2012distributed} equipped with \textit{clipped innovations}.

\begin{remark}
DMED is similar to the SAGE algorithm in \cite{chen2019resilient}. The computation goal in this paper, however, is different from that of~\cite{chen2019resilient}. SAGE is an algorithm for resilient distributed \textit{estimation}, where a network of agents measure the \textit{same} underlying parameter $\theta$ and attempt to recover its value from these measurements. That is, in~\cite{chen2019resilient}, the value of $\theta$ is the same for each agent. In contrast, this paper addresses distributed dynamic median consensus, where the agents have \textit{different} local values of $\theta_n$. As such, the analysis of the convergence of the DMED algorithm requires new techniques not found in~\cite{chen2019resilient}, which we will present in Section~\ref{sect: proof}.
\end{remark}

\begin{remark}
	The DMED algorithm essentially mimics the behavior of decentralized subgradient descent with errors for an $l_1$ minimization problem $\min_{\theta} \sum_{n=1}^{N}\abs{\theta_n - \theta}$ whose optimal solution set is $\Theta$. 
\end{remark}

We present the main result on the convergence of DMED.
\begin{theorem}\label{thm:main}
	Under Assumptions \ref{ass:network}-\ref{ass:noise}, the local median estimate of every agent $n \in [N]$ in Algorithm~\eqref{al:each} converges to $\Theta$ a.s. in that $\mathbb{P}\left(\lim_{t \rightarrow \infty } (t+1)^{\tau_3} \text{dist}(x_n(t), \Theta) = 0\right) = 1$ for all $n$ simultaneously. 
\end{theorem}

Theorem~\ref{thm:main} states that, when $N$ is odd, i.e., when the median is unique, all local estimates simultaneously converge to the same unique median almost surely. When $N$ is even, simultaneously, the distances between each local estimate and the set of medians will converge to 0 almost surely. However, when $N$ is even, local estimates are not guaranteed to converge to a particular point in the set of medians in that the estimates may wander within the set of medians asymptotically. Even when $N$ is even, as Lemma \ref{lm:aymcns} suggests, all local estimates are still guaranteed to reach consensus almost surely. The sample wise convergence rate we obtain is $\mathcal{O}\left(  (t+1)^{-\tau_3}\right)$; for instance, if $\abs{v_n(t)} \le (t+1)^{-1}$, $\delta_0$ can be taken to be 1, and by choosing $\tau_1 = 0.5, \tau_2 = 0.3$, we can set $\tau_3 = 0.4$, to ensure a $\mathcal{O}\left( (t+1)^{-0.4} \right)$ convergence rate.

\section{Proof of Theorem \ref{thm:main}}
\label{sect: proof}
To prove Theorem \ref{thm:main}, we first bound the local observation errors in Lemma \ref{lm:obsbd}, then bound the consensus errors in Lemma \ref{lm:aymcns}. Lemma \ref{lm:localcvg} shows that there exists a local contraction for the distance between network average and the the set of medians. Lemma \ref{lm:globalcvg} states that the network average will always enter the local contraction region. Combing these arguments we prove the theorem.

We use the following independent lemma to upper bound the local observation errors. 
\begin{lemma}
\label{lm:obsbd}
	Let $\{z_t\}$ be a $\mathbb{R}$ valued discrete time process
	\begin{equation}
	\label{eq:lm2}
	    z_{t+1} = (1 - r_1(t)) z_t + r_1(t)(r_2(t) + w(t))
	\end{equation}
	where the sequence $\{r_1(t)\}$ is deterministic with $r_1(t) = a_1(t+1)^{-\mu} \le 1$ for $a_1 > 0$, sequence $\{r_2(t)\}$ is almost surely bounded $\abs{r_2(t)} \le a_2(t+1)^{-\delta}$ for $a_2 > 0$, and $\{w(t)\}$ is i.i.d. random noise with $\mathbb{E}\left(w(t)\right) = 0$ and $\mathrm{Var}(w(t)) = \sigma^2 < \infty$. For $\delta \ge 1$, define $\delta_0 = 1 - \bar{\epsilon}$ for any $0 < \bar{\epsilon} < 1$ and take $\mu = \delta_0$; for $0 < \delta < 1$, define $\delta_0 = \delta$ and take $\mu$ such that $\delta \le \mu < 1$. Then, for any $0 < \epsilon_0 < \delta_0$, we have $\mathbb{P}\left(\lim_{t \rightarrow \infty} (t+1)^{\delta_0 - \epsilon_0} z_t^2 = 0\right) = 1$.
\end{lemma}
\begin{proof}
	 We consider sample paths where $\abs{r_2(t)} \le a_2(t+1)^{-\delta}$ holds for all $t$. We first show that $\mathbb{E}\left(|z_{t}|\right) < \infty$. Taking absolute value on the recursion~\eqref{eq:lm2} gives
	 
   \begin{equation*}
    \abs{z_{t+1}}
    \le (1 - \frac{a_1}{(t+1)^{\mu}})\abs{z_t} + \frac{a_1}{(t+1)^{\mu}}\left[\frac{a_2}{(t+1)^{\delta}} + \abs{w(t)}\right].
   \end{equation*}
   By Jensen's inequality we have $\mathbb{E}\left(|w(t)|\right) \le \sqrt{\mathbb{E}\left(w^2(t)\right)} = \sigma$. Then, there exists some constant
   $c_1 > a_1 \delta$ such that for sufficiently large $t$, 
   \begin{equation*}
      \mathbb{E}\left(\abs{z_{t+1}}\right) \le (1- \frac{a_1}{(t+1)^{\mu}}) \mathbb{E}\left(\abs{z_t}\right) + \frac{c_1}{(t+1)^{\mu}}.
   \end{equation*}
   By Lemma 4.1 in \cite{kar2013distributed} we have $\mathbb{E}\left(\abs{z_{t}}\right) < \infty$ for all $t$. Given the independence condition of $\{w(t)\}$, for sufficiently large $t$, there exist constants $c_2, c_3 > 0$ such that
   \begin{equation}
   \label{eq:exprecursion}
   \begin{split}
       \mathbb{E}\left(z^2_{t+1}\right)
       & \le \mathbb{E}\left( (1 - \frac{a_1}{(t+1)^{\mu}}) z_t + \frac{a_1 a_2}{(t+1)^{\mu + \delta}}\right)^2 \\
       & \quad + \frac{a_1^2}{(t+1)^{2\mu}} \mathbb{E}\left(w^2(t)\right) \\
       & \le (1 - \frac{a_1}{(t+1)^{\mu}})^2 \mathbb{E}\left(z_t^2\right) + \frac{a_1^2 a_2^2}{(t+1)^{2\mu + 2\delta}} \\
       & \quad +  (1 - \frac{a_1}{(t+1)^\mu} )\frac{2a_1 a_2}{(t+1)^{\mu+\delta}} \mathbb{E}\left(\abs{z_t}\right)  \\
       & \quad + \frac{a_1^2 \sigma^2}{(t+1)^{2\mu}} \\
       & \le (1 - \frac{c_2}{(t+1)^{\mu}}) \mathbb{E}\left(z_t^2\right) + \frac{c_3}{(t+1)^{\mu + \delta_0}}. 
    \end{split}
   \end{equation}
   The last inequality is due to the definition of $\delta_0$ which implies $\mu + \delta_0 = \min\{2\mu, \mu + \delta\}$ and $0 < \mu < 1$. By Lemma 4.2 in \cite{kar2013distributed}, relation (\ref{eq:exprecursion}) leads to that for any $0 < \epsilon_0 < \delta_0$, we have
   \begin{equation}
   \label{eq:ztmsecvg}
       \lim_{t \rightarrow \infty} (t+1)^{\delta_0 - \epsilon_0} \mathbb{E}\left(z_{t}^2\right) = 0. 
   \end{equation}
   Now we fix $\epsilon_0$. The definition of $\delta_0$ implies $0 < \delta_0 < 1$ and thus $(t+1)^{\delta_0 - \epsilon_0}$ is concave in $t$. Hence, we have
   \begin{equation*}
       (t+2)^{\delta_0 - \epsilon_0} \le (t+1)^{\delta_0 - \epsilon_0}[1 + (\delta_0 - \epsilon_0) (t+1)^{-1}].
   \end{equation*}
   Combining this fact with relation (\ref{eq:exprecursion}) gives that for sufficiently large $t$,
   \begin{equation}
   \label{eq:vt-med-bound}
   \begin{split}
     & (t+2)^{\delta_0 - \epsilon_0}\mathbb{E}\left(z_{t+1}^2\right) \\
     & \le \left[1 - \frac{c_2}{(t+1)^\mu} + \frac{\delta_0 - \epsilon_0}{t+1} - \frac{c_2(\delta_0 - \epsilon_0)}{(t+1)^{\mu + 1}}\right] \\
     & \quad \cdot (t+1)^{\delta_0 - \epsilon_0} \mathbb{E}\left(z_t^2\right) + \frac{c_3}{(t+1)^{\mu + \epsilon_0}}\left(1 + \frac{\delta_0 - \epsilon_0}{t+1}\right) \\
     & \le (1 - \frac{c_5}{(t+1)^{\mu}}) (t+1)^{\delta_0 - \epsilon_0} \mathbb{E}\left(z_t^2\right) + \frac{c_6}{(t+1)^{\mu + \epsilon_0}}
   \end{split}
   \end{equation}
    for some constants $c_5, c_6 > 0$. Define the process
    \begin{equation*}
    \begin{split}
      V(t) & = (t+1)^{\delta_0 - \epsilon_0} z_t^2 \\
      & \quad - \sum_{i=0}^{t-1}\left[\left(\Pi_{j = i + 1}^{t - 1}(1 - \frac{c_5}{(j+1)^{\mu}})\right) \frac{c_6}{(i+1)^{\mu + \epsilon_0}}\right].
    \end{split}
    \end{equation*}
    An application of Lemma 25 in \cite{kar2012distributed} leads to 
    \begin{equation}
    \label{eq:limit0}
      \lim_{t \rightarrow \infty} \sum_{i=0}^{t-1}\left[\left(\Pi_{j = i + 1}^{t - 1}(1 - \frac{c_5}{(j+1)^{\mu}})\right) \frac{c_6}{(i+1)^{\mu + \epsilon_0}}\right] = 0.
    \end{equation}
    Also note that we can split
    \begin{equation*}
    \begin{split}
       & \sum_{i=0}^{t}\left[\left(\Pi_{j = i + 1}^{t}(1 - \frac{c_5}{(j+1)^{\mu}})\right) \frac{c_6}{(i+1)^{\mu + \epsilon_0}}\right] \\
       &  = \left[1 - \frac{c_5}{(t+1)^{\mu}}\right] \sum_{i=0}^{t-1}\left[\left(\Pi_{j = i+1}^{t - 1}(1 - \frac{c_5}{(j+1)^{\mu}})\right)\frac{c_6}{(i+1)^{\mu + \epsilon_0}}\right] \\
       & \quad + \frac{c_6}{(t+1)^{\mu + \epsilon_0}}.
    \end{split}
    \end{equation*}
    Denote the filtration $\mathcal{F}_z (t)$ the natural filtration of the process $\{(t+1)^{\delta_0 - \epsilon_0}z_t^2\}$, and note that $V(t)$ is adapted to this filtration. Then, by the independence condition,
    \begin{equation*}
    \begin{split}
      & \mathbb{E}\left(V(t+1) \mid \mathcal{F}_z(t)\right) \\
      & = \mathbb{E}\left((t+2)^{\delta_0 - \epsilon_0} z_{t+1}^2\mid \mathcal{F}_z (t)\right) \\
      & \quad - \sum_{i=0}^{t}\left[\left(\Pi_{j = i + 1}^{t}(1 - \frac{c_5}{(j+1)^{\mu}})\right) \frac{c_6}{(i+1)^{\mu + \epsilon_0}}\right] \\
      & \le [1 - \frac{c_5}{(t+1)^{\mu}}] (t+1)^{\delta_0 - \epsilon_0} \mathbb{E}\left(z_t^2\right) + \frac{c_6}{(t+1)^{\mu + \epsilon_0}} \\
      & \quad - \sum_{i=0}^{t}\left[\left(\Pi_{j = i + 1}^{t}(1 - \frac{c_5}{(j+1)^{\mu}})\right) \frac{c_6}{(i+1)^{\mu + \epsilon_0}}\right] \\
      & = (1 - \frac{c_5}{(t+1)^{\mu}}) V(t) \\
      & \le V(t).
    \end{split}
    \end{equation*}
    Therefore, $\{V(t)\}$ is a supermartingale. By \eqref{eq:limit0}, $\{V(t)\}$ is bounded below. It follows that there exists a finite random variable $V_*$ such that $\lim_{t \rightarrow \infty} V(t) = V_*$ almost surely. Thus, we have $\lim_{t \rightarrow \infty} (t+1)^{\delta_0 - \epsilon_0} z_t^2 = V_*$ almost surely. Then, by Fatou's lemma and \eqref{eq:ztmsecvg} we have 
    \begin{equation*}
      0 \le \mathbb{E}\left(\lim_{t \rightarrow \infty} (t+1)^{\delta_0 - \epsilon_0} z_t^2\right) \le \liminf\limits_{t \rightarrow \infty} (t+1)^{\delta_0 - \epsilon_0}\mathbb{E}\left(z_t^2\right) = 0.
    \end{equation*}
    Thus, we have $\mathbb{P}\left(\lim_{t \rightarrow \infty} (t+1)^{\delta_0 - \epsilon_0} z_t^2 = 0 \right) = 1$.
\end{proof}

We study the behavior of \textit{all} of the agents' estimates together, so, for convenience, we introduce the following notation. Let $K_t = \text{diag}([k_1(t), \ldots, k_N(t)]) \in \mathbb{R}^{N \times N}, \mathbf{x}(t) = [x_1(t), \ldots, x_N(t)]^\top \in \mathbb{R}^N$, $\bar{\pmb{\theta}}(t) = [\bar{\theta}_1(t), \ldots, \bar{\theta}_N(t)]^\top \in \mathbb{R}^N$, then \eqref{al:each} can be rewritten as 
\begin{equation}
\label{al:vector}
	\mathbf{x}(t+1) = (I - \beta_t L(t)) \mathbf{x}(t) - \alpha_t K_t(\mathbf{x}(t) - \bar{\pmb{\theta}}(t)).
\end{equation} 

The following result (from~\cite{chen2019resilient}) establishes that the agents reach consensus. Define $P_N = N^{-1}\mathbf{1}_N \mathbf{1}_N^\top$. 
\begin{lemma}[Lemma 1 in \cite{chen2019resilient}]
\label{lm:aymcns} 
	Under Assumption \ref{ass:network}, for every $0 \le \epsilon_1 < \tau_1 - \tau_2 + \tau_3$, the iterates $\mathbf{x}(t)$ in (\ref{al:vector}) satisfies $\mathbb{P}\left(\lim_{t \rightarrow \infty} (t + 1)^{\tau_1 - \tau_2 + \tau_3 - \epsilon_1} \|\mathbf{x}(t) - P_N \mathbf{x}(t)\|_2 = 0\right) = 1$.
\end{lemma}

Let $\bar{x}(t) = N^{-1}\mathbf{1}_{N}^\top \mathbf{x}(t)$ denote the network average at time $t$. The following lemma analyze the local contraction of $\bar{x}(t)$.
\begin{lemma}
\label{lm:localcvg}
	Define the auxiliary threshold
	\begin{equation*}
		\bar{\gamma}_t = \gamma_t - \frac{c_{\delta}}{(t+1)^{\delta_1}} - \frac{c_{\eta}}{(t+1)^{\eta}}
	\end{equation*}
	where
	\begin{equation*}
		\delta_1 = 0.5(\delta_0 - \epsilon_0), 
		\eta = \tau_1 - \tau_2 + \tau_3 - \epsilon_1
	\end{equation*}
	for arbitrarily small
	\begin{equation*}
		0 < \epsilon_0 < \delta_0, 0 < \epsilon_1 < \tau_1 - \tau_2.
	\end{equation*}
	Then, almost surely, there exists a finite $T_0$ and positive constants $c_{\delta}, c_{\eta}$ such that if $\text{dist}(\bar{x}(T_1), \Theta) \le \bar{\gamma}_{T_1}$ for some $T_1 \ge T_0$, then $\text{dist}(\bar{x}(t), \Theta) \le \bar{\gamma}_t$ for all $t \ge T_1$. 
\end{lemma}

\begin{proof}
Substitute \eqref{eq:obsmodel} into \eqref{eq:localavg}, we obtain, for each $n \in [N]$,
\begin{equation*}
\begin{split}
	\bar{\theta}_n(t) - \theta_n & = (1 - \frac{c_\mu}{(t+1)^\mu})(\bar{\theta}_n(t) - \theta_n) \\
	& \quad + \frac{c_\mu}{(t+1)^\mu}(w_n(t) + \nu_n(t)).
\end{split}
\end{equation*}
Let $z_t = \bar{\theta}_n(t) - \theta_n$ and by applying Lemma \ref{lm:obsbd} for each $n \in [N]$
\begin{equation}
\label{eq:as1}
	\mathbb{P}\left(\lim_{t \rightarrow \infty}(t+1)^{\delta_0 - \epsilon_0}\abs{\bar{\theta}_n(t) - \theta_n}^2 = 0\right) = 1. 
\end{equation}
By Lemma \ref{lm:aymcns} we have for all $n \in [N]$,
\begin{equation}
\label{eq:as2}
	\mathbb{P}\left(\lim_{t \rightarrow \infty}(t+1)^{\eta} \abs{\bar{x}(t) - x_n(t)} = 0\right) = 1.
\end{equation}
We perform the derivations in a sample path $\omega \in \Omega$ such that there exist positive constants $c_{\delta, \omega}, c_{\eta, \omega}$ such that
\begin{equation}
\label{eq:errs_bounds}
\begin{split}
    \abs{\bar{\theta}_n(t, \omega) - \theta_n} \le \frac{c_{\delta, \omega}}{(t+1)^{\delta_1}}, \\ \abs{\bar{x}(t, \omega) - x_n(t, \omega)} \le \frac{c_{\eta, \omega}}{(t+1)^{\eta}}
\end{split}
\end{equation}
hold for all $n \in [N]$. As a consequence of~\eqref{eq:as1}\eqref{eq:as2}, the set of all such sample paths has probability $1$. Define
\begin{equation}
\label{eq:defen}
    e_n(t, \omega) \triangleq x_n(t, \omega) - \bar{x}(t, \omega) + \theta_n - \bar{\theta}_n(t, \omega).
\end{equation}
Then,
\begin{equation}
\label{eq:all_errs_bounds}
\abs{e_n(t, \omega} \le c_{\delta, \omega}(t+1)^{-\delta_1} + c_{\eta, \omega}(t+1)^{-\eta}.
\end{equation}
Now that we have bounded the consensus errors and the local observation errors, we start to analyze the dynamics of the network average $\bar{x}(t, \omega)$. Multiplying $N^{-1} \mathbf{1}^\top$ on both sides of (\ref{al:vector}) leads to
\begin{equation}
\label{eq:avgdym}
	\bar{x}(t+1, \omega) = \bar{x}(t, \omega) -  \frac{\alpha_t}{N} \sum_{n=1}^{N}k_n(t)(x_n(t, \omega) - \bar{\theta}_n(t, \omega)). 
\end{equation}

\textit{Step 1}, we analyze the net effect of agent $n$ where $\theta_n \notin \Theta$, i.e., $\sum_{\theta_n \not \in \Theta} k_n(t)(x_n(t, \omega) - \bar{\theta}_n(t, \omega)$. Since $\tau_3 < \delta_1, \tau_3 < \eta$, we can take $t_1$ as the least $t$ such that $\bar{\gamma}_{t, \omega} > 0$. Also, recall the definition of $d_{\min}$ in (\ref{eq:dmin}), we take $t_2$ as the least $t \ge t_1$ such that $d_\min > 2\gamma_t$ and $\alpha_t < N$. For $t \ge t_2, \theta_n \notin \Theta$, if $\text{dist}(\bar{x}(t, \omega), \Theta) \le \bar{\gamma}_{t, \omega}$, we have
\begin{equation}
\label{eq:eachlb}
\begin{split}
  & \abs{x_n(t, \omega) - \bar{\theta}_n(t, \omega)} \\
  & \ge \abs{\bar{x}(t, \omega) - \theta_n} - \abs{e_n(t, \omega)}\\
  & \ge \text{dist}(\theta_n, \Theta) - \text{dist}(\bar{x}(t, \omega), \Theta) - \abs{e_n(t, \omega)} \\
  & \ge d_\min - \gamma_t > \gamma_t.
\end{split}
\end{equation}
Since $x_n(t, \omega) - \bar{\theta}_n(t, \omega) = e_n(t, \omega) + \bar{x}(t, \omega) - \theta_n$, and by the definition of $t_1$ we have $\abs{e_n(t, \omega)} < \gamma_t$. Combining this with \eqref{eq:eachlb} we have
\begin{equation}
\label{eq:equal-sign}
    k_n(t)(x_n(t, \omega) - \bar{\theta}_n(t, \omega)) = \gamma_t \text{sign}(\bar{x}(t, \omega) - \theta_n).
\end{equation}
By the hypothesis, $\text{dist}(\bar{x}(t, \omega), \Theta) \le \bar{\gamma}_t < \gamma_t < 0.5d_\min$, 
\begin{equation}
\label{eq:in-between}
\begin{split}
    \theta_{\frac{N-1}{2}} < \bar{x}(t, \omega) < \theta_{\frac{N+3}{2}}, & \  N \text{ is odd}, \\
\theta_{\frac{N}{2} - 1} < \bar{x}(t, \omega) < \theta_{\frac{N}{2}+2}, &\ N \text{ is even}.
\end{split}
\end{equation}
Then, by the definition of median we have zero net effect from $\theta_n \notin \Theta$, i.e.,
\begin{equation*}
  \sum_{\theta_n \notin \Theta} k_n(t, \omega)(x_n(t, \omega) - \theta_n) = \sum_{\theta_n \notin \Theta} \gamma_t \text{sign}(\bar{x}(t, \omega) - \theta_n) = 0.
\end{equation*}

\textit{Step 2}, we analyze the effects of all agents $n$ with $\theta_n \in \Theta$, i.e., $\sum_{\theta_n \in \Theta} k_n(t)(x_n(t, \omega) - \bar{\theta}_n(t, \omega)$. We define $m$ as the index that $\theta_m \in \Theta$ and $\text{dist}(\bar{x}(t, \omega), \Theta) = |\bar{x}(t, \omega) - \theta_m|$ if it exists. If $N$ is odd, $m = (N+1)/2$; if $N$ is even and $\bar{x}(t, \omega) \notin \Theta$, $m = N/2$ or $N/2+1$. By the hypothesis and \eqref{eq:all_errs_bounds},
\begin{equation*}
\begin{split}
     |x_{m}(t, \omega) - \bar{\theta}_{m}(t, \omega)| 
     & \le \abs{\bar{x}(t, \omega) - \theta_{m}} + e_{m}(t, \omega) \\
     & \le \bar{\gamma}_{t, \omega} + |e_m(t, \omega | \le \gamma_t.
\end{split}
\end{equation*}
Thus, $k_{m}(t, \omega) = 1$. We next discuss all three cases: $N$ is odd; $N$ is even and $\bar{x}(t, \omega) \notin \Theta$; $N$ is even but $\bar{x}(t, \omega) \in \Theta$. 

\textit{Step 2a}, when $N$ is odd, by step I, (\ref{eq:avgdym}) reduces to
\begin{equation*}
  \bar{x}(t+1, \omega) = \bar{x}(t, \omega) - \frac{\alpha_t}{N}\left(x_{\frac{N+1}{2}}(t, \omega) - \bar{\theta}_{\frac{N+1}{2}}(t, \omega)\right). 
\end{equation*}
It follows that 
\begin{equation}
\begin{split}
\label{eq:avg-contr-odd}
  &  \text{dist}(\bar{x}(t+1, \omega), \Theta) \\
  & = |\bar{x}(t, \omega) - \theta_{\frac{N+1}{2}} - \frac{\alpha_t}{N} [x_{\frac{N+1}{2}}(t, \omega) - \bar{\theta}_{m}(t, \omega)]| \\
  & \le (1 - \frac{\alpha_t}{N})\abs{\bar{x}(t, \omega) - \theta_{\frac{N+1}{2}}} + \frac{\alpha_t}{N}\abs{e_{\frac{N+1}{2}}(t, \omega)} \\
  & \le (1 - \frac{\alpha_t}{N})\bar{\gamma}_{t, \omega} + \frac{\alpha_t}{N} \frac{c_{\delta, \omega}}{(t+1)^{\delta_1}} + \frac{\alpha_t}{N} \frac{c_{\eta, \omega}}{(t+1)^{\eta}}.
\end{split}
\end{equation}
We next show $\text{dist}(\bar{x}(t+1), \omega), \Theta) \le \bar{\gamma}_{t+1, \omega}$. Define
\begin{equation*}
    \Delta_t = (t+1)^{\tau_3}\bar{\gamma}_{t, \omega} = \gamma_0 - \frac{c_{\delta, \omega}}{(t+1)^{\delta_1 - \tau_3}} - \frac{c_{\eta, \omega}}{(t+1)^{\eta - \tau_3}}. 
\end{equation*}
Since $\tau_3 < \delta_1, \tau_3 < \eta$, $\Delta_t$ is increasing with $t$ and thus 
\begin{equation*}
  \bar{\gamma}_{t+1, \omega} = \frac{\Delta_{t+1}}{(t+2)^{\tau_3}} \ge \frac{\Delta_{t}}{(t+2)^{\tau_3}} = \left(\frac{t+1}{t+2}\right)^{\tau_3} \bar{\gamma}_{t, \omega}.
\end{equation*}
To show that $\text{dist}(\bar{x}(t+1), \omega), \Theta) \le \bar{\gamma}_{t+1, \omega}$, it suffices to show that 
\begin{equation*}
   (1 - \frac{\alpha_t}{N})\bar{\gamma}_{t, \omega} + \frac{\alpha_t}{N} \frac{c_{\delta, \omega}}{(t+1)^{\delta_1}} + \frac{\alpha_t}{N} \frac{c_{\eta, \omega}}{(t+1)^{\eta}} \le \left(\frac{t+1}{t+2}\right)^{\tau_3} \bar{\gamma}_{t, \omega},  
\end{equation*}
rearranging the above relation gives that
\begin{equation*}
\begin{split}
    1 - \frac{\alpha_t}{N} \Bigg( \underbrace{1 - \frac{c_{\delta, \omega}}{\Delta_t (t + 1)^{\delta_1 - \tau_3}} - \frac{c_{\eta, \omega}}{\Delta_t(t+1)^{\eta - \tau_3}}}_{\mathcal{T}_1}\Bigg) \le \left(\frac{t+1}{t+2}\right)^{\tau_3}.
\end{split}
\end{equation*}
As $\mathcal{T}_1$ is increasing in $t$, we can take $t_3$ as the least $t \ge t_2$ such that $\mathcal{T}_1 \ge 0$. Since $1 - x \le e^{-x}$ for $x \ge 0$, it suffices to show 
\begin{equation*}
\begin{split}
    \frac{\alpha_t}{N \tau_3}\left(\frac{c_{\delta, \omega}}{\Delta_t(t+1)^{\delta_1 - \tau_3}} + \frac{c_{\eta, \omega}}{\Delta_t(t+1)^{\eta - \tau_3}}  - 1\right)  \le \ln \left(\frac{t + 1}{t+2}\right).
\end{split}
\end{equation*}
Take $t_4$ as the least $t \ge t_3$ such that $\Delta_t \ge \gamma_0/2$. Since $\ln \left(\frac{t + 1}{t+2}\right) \ge 1 - \frac{t+2}{t+1} = -\frac{1}{t+1}$, for $t \ge t_4$, it suffices to show 
\begin{equation*}
  \frac{\alpha_t}{N \tau_3}\left(\frac{2c_{\delta, \omega}}{\gamma_0(t+1)^{\delta_1 - \tau_3}} + \frac{2c_{\eta, \omega}}{\gamma_0 (t+1)^{\eta - \tau_3}}  - 1\right) \le -\frac{1}{t+1},
\end{equation*}
which rearranges to 
\begin{equation*}
\begin{split}
    & \frac{2c_{\delta, \omega}}{\gamma_0 (t+1)^{ \delta_1 - \tau_3}} + \frac{2c_2}{\gamma_0 (t+1)^{\eta - \tau_3}} + \frac{N \tau_3}{\alpha_0(t+1)^{1 - \tau_1}}  \le 1.
\end{split}
\end{equation*}
Since the left hand side is monotonically decreasing to 0, we can take $t_5$ as the least $t \ge t_4$ such that the above relation holds. Taking $T_0 = t_5$ completes this case.

\textit{Step 2b}: when $N$ is even and $\bar{x}(t, \omega) \notin \Theta$, without loss of generality, we consider $\bar{x}(t, \omega) < \theta_{N/2}$, then for $t \ge t_2$
\begin{equation*}
\begin{split}
  & \abs{x_{\frac{N}{2} + 1}(t, \omega) - \bar{\theta}_{\frac{N}{2} + 1}(t, \omega)} \\
  & \ge \abs{\bar{x}(t, \omega) - \theta_{\frac{N}{2} + 1}} - \abs{e_{\frac{N}{2} + 1}(t, \omega)} \\
  & \ge \abs{\theta_{\frac{N}{2}} - \theta_{\frac{N}{2}+1}} - \abs{\bar{x}(t, \omega) - \theta_{\frac{N}{2}}} -\abs{e_{\frac{N+1}{2}}(t, \omega)}\\
  & \ge \abs{\theta_{\frac{N}{2}} - \theta_{\frac{N}{2}+1}} - \gamma_t \ge \gamma_t.
\end{split}
\end{equation*}
Hence, $k_{N/2 + 1}\left(t, \omega)(x_{N/2 + 1}(t, \omega) - \bar{\theta}_{N/2 + 1}(t, \omega)\right) = - \gamma_{t}$ by the same argument in \eqref{eq:eachlb}\eqref{eq:equal-sign}. Then, (\ref{eq:avgdym}) reduces to
\begin{equation}
\label{eq:even_avg}
  \bar{x}(t+1, \omega) = \bar{x}(t, \omega) - \frac{\alpha_t}{N}\left(x_{\frac{N}{2}}(t, \omega) - \bar{\theta}_{\frac{N}{2}}(t, \omega) - \gamma_t\right).
\end{equation}
It follows that $\abs{\bar{x}(t+1, \omega) - \bar{x}(t, \omega)} \le 2N^{-1} \alpha_t \gamma_t$. Take $t_6$ as the least $t \ge t_2$ such that $2N^{-1} \alpha_t \gamma_t < \abs{\theta_{N/2} - \theta_{N/2+1}}$. Then, for $t \ge t_6$, we either have $\bar{x}(t+1, \omega) \in \Theta$ or $\text{dist}(\bar{x}(t+1, \omega), \Theta) = \abs{\bar{x}(t+1, \omega) - \theta_{N/2}} > 0$. In the first case, $\text{dist}(\bar{x}(t+1, \omega), \Theta) = 0$.  In the second case, $\bar{x}(t+1, \omega) < \theta_{N/2}$ and
\begin{equation}
\label{eq:avg_cvg_even}
\begin{split}
  & \text{dist}(\bar{x}(t+1, \omega), \Theta) 
  \\
  & \le \abs{\bar{x}(t, \omega) - \frac{\alpha_t}{N}\left(-\gamma_t + x_{\frac{N}{2}}(t, \omega) - \bar{\theta}_{\frac{N}{2}}(t, \omega)\right) - \theta_{\frac{N}{2}}} \\
  & \le \abs{(1 -  \frac{\alpha_t}{N})(\bar{x}(t, \omega) - \theta_{\frac{N}{2}}) + \frac{\alpha_t \gamma_t}{N}} + \frac{\alpha_t}{N}\abs{e_{\frac{N}{2}}(t, \omega)} \\
  & \le \max\{(1 -  \frac{\alpha_t}{N}) \abs{\bar{x}(t, \omega) - \theta_{\frac{N}{2}}}, \frac{\alpha_t \gamma_t}{N}\} + \frac{\alpha_t}{N}\abs{e_{\frac{N}{2}}(t, \omega)}.
\end{split}
\end{equation}
If $(1 -  N^{-1}\alpha_t)\abs{\bar{x}(t, \omega) - \theta_{N/2}} \ge N^{-1} \alpha_t \gamma_t$, the above relation falls into the same pursuit of (\ref{eq:avg-contr-odd}) and thus there exists $t_{7} \ge t_6$ such that $\text{dist}(\bar{x}(t+1, \omega), \Theta) \le \bar{\gamma}_{t+1, \omega}$ for $t \ge t_7$. Otherwise, to show $\text{dist}(\bar{x}(t+1, \omega), \Theta) \le \bar{\gamma}_{t+1, \omega}$ it suffices to show 
\begin{equation*}
\begin{split}
  \frac{\alpha_t \gamma_t}{N} + \frac{\alpha_t}{N}\abs{e_{\frac{N}{2}}(t, \omega)} \le \bar{\gamma}_{t+1, \omega}. 
\end{split}
\end{equation*}
Substitute \eqref{eq:all_errs_bounds} into the display above, it suffices to show
\begin{equation*}
\begin{split}
   & \frac{\alpha_0 \gamma_0}{(t+1)^{\tau_1}} +  \frac{\alpha_0 c_{\eta, \omega}}{(t+1)^{\eta + \tau_1 - \tau_3}} + \frac{\alpha_0 c_{\delta, \omega}}{(t+1)^{\delta_1 + \tau_1 - \tau_3}} + \frac{N c_{\eta, \omega}}{(t+2)^{\eta - \tau_3}} \\
   & + \frac{N c_{\delta, \omega}}{(t+2)^{\delta_1 - \tau_3}} \le N \gamma_0, 
\end{split}
\end{equation*}
The left hand side is decreasing to 0 as $\tau_3 < \delta_1, \tau_3 < \eta$, so we can take $t_8$ as the least $t \ge t_7$ such that the above relation holds. Taking $T_0$ as $t_8$ addresses this case. 

\textit{Step 2c}, when $N$ is even and $\bar{x}(t, \omega) \in \Theta$, we show that $\bar{x}(t+1, \omega) \in \Theta$. In this case, \eqref{eq:avgdym} reduces to
\begin{equation}
\label{eq:part3}
\begin{split}
    & \bar{x}(t+1, \omega) - \bar{x}(t, \omega) \\
	& = - \frac{\alpha_t}{N} k_{\frac{N}{2}} (t) \underbrace{ (x_{\frac{N}{2}}(t, \omega) - \bar{\theta}_{\frac{N}{2}}(t, \omega))}_{\mathcal{T}_2} \\
	& \quad - \frac{\alpha_t}{N} k_{\frac{N}{2}+1} (t) \underbrace{ (x_{\frac{N}{2}+1}(t, \omega) - \bar{\theta}_{\frac{N}{2}+1}(t, \omega))}_{\mathcal{T}_3}.
\end{split}
\end{equation}
We discuss two possibilities, one is that $|\T_2| \le \gamma_t $ or $|\T_3| \le \gamma_t$, the other is both $|\T_2 > \gamma_t$ and $|\T_2| > \gamma_t$. First, consider $\abs{\mathcal{T}_2} \le \gamma_t$ ($|\T_3| \le \gamma_t$ is a symmetric case). Since
\begin{equation*}
\begin{split}
    \abs{\bar{x}(t, \omega) - \theta_{\frac{N}{2}}} 
    & \le \abs{\T_2} + \abs{e_{\frac{N}{2}+1}(t, \omega)} \\
    & \le \gamma_t + \frac{c_{\delta, \omega}}{(t+1)^{\delta_1}} + \frac{c_{\eta, \omega}}{(t+1)^{\eta}},
\end{split}
\end{equation*}
and $\bar{x}(t, \omega) \in \Theta$, we obtain
\begin{equation*}
\begin{split}
    -\T_3
    & =  \theta_{\frac{N}{2}+1} - \bar{x}(t, \omega) - e_{\frac{N}{2}+1}(t, \omega) \\
    & \ge \theta_{\frac{N}{2} + 1} - \theta_{\frac{N}{2}} - \abs{\bar{x}(t, \omega) - \theta_{\frac{N}{2}}} - \abs{ e_{\frac{N}{2}+1}(t, \omega)} \\
    & \ge d_\min - \gamma_t - \frac{2c_{\delta, \omega}}{(t+1)^{\delta_1}} - \frac{2c_{\eta, \omega}}{(t+1)^{\eta}}.
\end{split}
\end{equation*}
Take $t_9$ as the least $t \ge t_8$ so that the right hand side on the last line of the above relation is larger than $\gamma_t$. Then, substituting the values of $k_{N/2}(t)$ and $k_{N/2+1}(t)$ into \eqref{eq:part3} we have $\bar{x}(t+1, \omega) \ge \bar{x}(t, \omega)$. Take $t_{10}$ as the least $t \ge t_9$ such that $2N^{-1}\alpha_t \gamma_t \le d_\min - (\gamma_t + c_{\delta, \omega}(t+1)^{-\delta_1} + c_{\eta, \omega}(t+1)^{-\eta})$, then by \eqref{eq:poss1-bound}, for $t \ge t_{10}$ we have
\begin{equation}
\label{eq:poss1-bound}
\begin{split}
    & \bar{x}(t+1, \omega) - \bar{x}(t, \omega) \\
    & \le \frac{2\alpha_t \gamma_t}{N} \le \theta_{\frac{N}{2}+1} - \theta_{\frac{N}{2}} - \abs{\bar{x}(t, \omega) - \theta_{\frac{N}{2}}} \\
    & \le \theta_{\frac{N}{2}+1} - \bar{x}(t, \omega), 
\end{split} 
\end{equation}
so we have $\bar{x}(t, \omega) \le \bar{x}(t+1, \omega) \le \theta_{N/2+1}$, and thus $\bar{x}(t+1, \omega) \in \Theta$. Second, if both $\abs{\T_2} > \gamma_t$ and $\abs{\T_3} > \gamma_t$, where it is easy to check that under $t \ge t_{10}$, $\T_2$ and $\T_3$ are of opposite signs and thus $\bar{x}(t+1, \omega) = \bar{x}(t, \omega) \in \Theta$. Suppose both $\T_2, \T_3 < -\gamma_t$, then 
\begin{equation*}
\begin{split}
    \bar{x}(t, \omega) - x_{\frac{N}{2}} 
    & = \T_2 - e_{\frac{N}{2}}(t, \omega) \\
    & < -\gamma_t +  \frac{c_{\delta, \omega}}{(t+1)^{\delta_1}} + \frac{c_{\eta, \omega}}{(t+1)^{\eta}}\\
    & = -\bar{\gamma}_t < 0,
\end{split}
\end{equation*}
which contradicts with $\bar{x}(t, \omega) \in \Theta$. Similar argument applies for $\T_2, \T_3 > \gamma_t$. Thus, taking $T_0 = t_{10}$ addresses this case.

\textit{Step 3}, in all there cases, in the sample path $\omega$ such that \eqref{eq:all_errs_bounds} holds, there exist some finite $T_0$ such that if for some $T_1 \ge T_0$ we have $\text{dist}(\bar{x}(T_1), \Theta) \le \bar{\gamma}_{T_1}$ for some $T_1 \ge T_0$, then $\text{dist}(\bar{x}(T_1 + 1), \Theta) \le \bar{\gamma}_{T_1 + 1}$. Since such sample paths $\omega$ has probability 1, $T_0$ exists a.s.
\end{proof}

The following lemma asserts that hypothesis of Lemma \ref{eq:localavg} will always be true, and thus establishes the global convergence of Algorithm $(\ref{al:each})$.
\begin{lemma}
\label{lm:globalcvg}
For $T_0, \bar{\gamma}_t$ established in Lemma \ref{lm:localcvg}, there exists a finite $T_1 \ge T_0$ such that $\text{dist}(\bar{x}(T_1), \Theta) \le \bar{\gamma}_{T_1}$.
\end{lemma}

\begin{proof}
We prove this lemma by contradiction. We work on sample paths $\omega$ as in Lemma \ref{lm:localcvg}. Suppose, on the contrary, for all $t \ge T_{0, \omega}, \ \text{dist}(\bar{x}(t, \omega), \Theta) > \bar{\gamma}_{t, \omega}$. We now show that this leads to  $\lim_{t\rightarrow\infty}\text{dist}(\bar{x}(t, \omega), \Theta)=-\infty$, a contradiction.

\textit{Step 1}, we first discuss the value of each \textit{clipped innovation} $k_n(t)[x(t, \omega) - \bar{\theta}_n(t, \omega)]$. We show that there exists some finite $T_2 \ge T_{0, \omega}$ such that $\forall t \ge T_2$, there exists at most one $n \in [N]$ such that $\abs{x_{n}(t, \omega) - \bar{\theta}_{n}(t, \omega)} \le \gamma_t$, which implies there exists at most one $n \in [N]$ such that $k_n(t) = 1$. Take $T_2$ as the smallest $t \ge T_0$ such that $2\gamma_t < d_{\min} - 2 c_{\delta}(t+1)^{-\delta_1} - 2 c_{\eta, \omega}(t+1)^{-\eta}.$
Suppose there exist two different $m, n \in [N]$ such that for both $i = m, n, \abs{x_{i}(t, \omega) - \bar{\theta}_{i}(t, \omega)} \le \gamma_t$ . Then, for $i = m, n$, 
\begin{equation*}
\begin{split}
  & \gamma_t \ge \abs{x_i(t, \omega) - \bar{\theta}_i(t, \omega)} \ge \abs{\bar{x}(t, \omega) - \theta_i} - \abs{e_i(t, \omega)}. 
\end{split}
\end{equation*}
Summing above relations for $i = m, n$, combining with (\ref{eq:all_errs_bounds}) we have $2\gamma_t \ge \abs{\theta_m - \theta_n} - 2 c_{\delta}{(t+1)^{-\delta_1}} - 2 c_{\eta, \omega}{(t+1)^{-\eta}}$, which contradicts with the choice of $T_2$. If it exists, for the agent $p$ that satisfies $\abs{x_{p}(t, \omega) - \bar{\theta}_{p}(t, \omega)} \le \gamma_t$ and $\theta_p \in \Theta$, take $T_3$ as the least $t \ge T_2$ such that $\bar{\gamma}_{t, \omega} \ge c_{\delta, \omega}(t+1)^{-\delta} + c_{\eta, \omega}(t+1)^{-\eta}$. Thus, by the contradiction hypothesis, 
\begin{equation}
\label{eq:signbound2}
\begin{split}
    \abs{e_p(t, \omega)} \le  \bar{\gamma}_{t, \omega} < \text{dist}(\bar{x}(t, \omega), \Theta) =  \abs{\bar{x}(t, \omega) - \theta_p}.
\end{split}
\end{equation}
Decomposing $x_p(t, \omega) - \bar{\theta}_p(t, \omega) = \bar{x}(t, \omega) - \theta_p + e_p(t, \omega)$ and combing with \eqref{eq:signbound2} we have 
\begin{equation}
\label{eq:closesign}
    \text{sign}(x_p(t, \omega) - \bar{\theta}_p(t, \omega)) = \text{sign}(\bar{x}(t, \omega) - \theta_p).
\end{equation}
For any agent $m$ such that $\abs{x_{m}(t, \omega) - \bar{\theta}_{m}(t, \omega)} > \gamma_t$, with the same reasoning in \eqref{eq:eachlb}\eqref{eq:equal-sign} we have
\begin{equation}
\label{eq:equal-sign2}
    k_m(x_m(t, \omega) - \bar{\theta}_m(t, \omega)) = \gamma_t \text{sign}(\bar{x}(t, \omega) - \theta_m).
\end{equation}

We next consider $t \ge T_3$ and assume $\bar{x}(t, \omega)$ is smaller than median(s) without loss of generality. We discuss all three possible cases.

\textit{Step 2a}, if there exists $p$ such that $\abs{x_{p}(t, \omega) - \bar{\theta}_p(t, \omega)} \le \gamma_t$ and $\theta_p \in \Theta$. Then,
\begin{equation*}
\begin{split}
    & \bar{x}(t, \omega) - \theta_{p-1} \\
    & =  x_p(t, \omega) - \bar{\theta}_p(t, \omega) + e_p(t, \omega) + \theta_p - \theta_{p - 1} > 0,
\end{split}
\end{equation*}
since by the definition of $T_2$,
\begin{equation*}
\begin{split}
    & \theta_p - \theta_{p-1} \ge d_{\min} > c_{\delta, \omega}(t+1)^{-\delta_1} + c_{\eta, \omega} (t+1)^{-\eta} + \gamma_t \\
    & \ge \abs{e_p(t, \omega)} + \abs{x_p(t, \omega) - \bar{\theta}_p(t, \omega)}.
\end{split}
\end{equation*}
Thus, we have $\theta_{p-1} < \bar{x}(t, \omega) < \theta_p$. By \eqref{eq:equal-sign2}, and the definition of median
\begin{equation}
\label{eq:net_effect}
\sum_{n \ne p}k_n(t, \omega)(x_n(t, \omega) - \bar{\theta}_n(t, \omega)) = 
\begin{cases}
0, & N \text{ is odd}, \\
-\gamma_t, & N \text{ is even}.
\end{cases}
\end{equation}
Thus, \eqref{eq:avgdym} reduces to
\begin{equation}
\label{eq:case1}
\begin{split}
  & \bar{x}(t+1, \omega) - \bar{x}(t, \omega)  \\
  & = 
    \begin{cases}
    - \frac{\alpha_t}{N} (x_p(t, \omega) - \bar{\theta}_p(t, \omega) - \gamma_t), &  N \text{ is even,} \\
    - \frac{\alpha_t}{N} (x_p(t, \omega) - \bar{\theta}_p(t, \omega) ), & N \text{ is odd.} 
  \end{cases} \\
  & \ge 0.
\end{split}
\end{equation}
where the last inequality follows from \eqref{eq:closesign} and $\bar{x}(t, \omega) < \theta_p$. By the contradiction hypothesis,
\begin{equation*}
\begin{split}
    & \abs{x_p(t, \omega) - \bar{\theta}_p(t, \omega)} \\
    & \ge \abs{\bar{x}(t, \omega) - \theta_p} - \abs{e_p(t, \omega} \\
    & > \bar{\gamma}_t - c_{\delta, \omega} (t+1)^{-\delta} - c_{\eta, \omega} (t+1)^{-\eta}, 
\end{split}
\end{equation*}
it follows from (\ref{eq:case1}) that for both even and odd $N$,
\begin{equation}
\label{eq:distdec1}
\begin{split}
	&  \bar{x}(t+1, \omega) - \bar{x}(t, \omega)  \\
	& \ge   \frac{\alpha_t}{N}(\gamma_t - \frac{2c_{\delta, \omega}}{(t+1)^{\delta_1}} - \frac{2c_{\eta, \omega}}{(t+1)^{\eta}}). 
\end{split}
\end{equation}

\textit{Step 2b}, consider the case that there exists one $q$ such that $\abs{x_{q}(t, \omega) - \bar{\theta}_q(t, \omega)} \le \gamma_t$, but $\theta_q \notin \Theta$. Since $\bar{x}(t, \omega)$ is smaller than the median(s), $\theta_q$ is also smaller than the median(s), otherwise $\abs{\bar{x}(t, \omega) - \theta_q } \ge d_\min > \gamma_t + c_{\delta, \omega}(t+1)^{-\delta} + c_{\eta, \omega}(t+1)^{-\eta}$ by the definition of $T_2$, and implies the contradiction that
\begin{equation*}
\begin{split}
    \abs{x_q(t, \omega) - \bar{\theta}_q(t, \omega)}  \ge \abs{\bar{x}(t, \omega) - \theta_q} - \abs{e_q(t, \omega)} > \gamma_t.
\end{split}
\end{equation*}
Then, from \eqref{eq:equal-sign2} 
\begin{equation}
\label{eq:case2}
\begin{split}
    & \bar{x}(t+1, \omega) - \bar{x}(t, \omega) \\
    & =   - \frac{\alpha_t \gamma_t}{N}\sum_{n \ne q} \text{sign}\left[x_n(t, \omega) - \bar{\theta}_n(t, \omega) \right] \\
    & \qquad - \frac{\alpha_t}{N} (x_q(t, \omega) - \bar{\theta}_q(t, \omega)) \\
    & = - \frac{\alpha_t \gamma_t}{N} \sum_{n \ne q} \text{sign}\left[ \bar{x}(t, \omega) - \theta_n \right] \\
    & \qquad - \frac{\alpha_t}{N} (x_q(t, \omega) - \bar{\theta}_q(t, \omega).
\end{split}
\end{equation}
Since $\bar{x}(t, \omega), \theta_q$ are smaller than the median(s), by counting $\theta_n$ on both sides of $\bar{x}(t, \omega)$ we have
\begin{align*}
\sum_{n \ne q} \text{sign}\left[ \bar{x}(t, \omega) - \theta_n \right] \le -2.
\end{align*}
Hence, it follows from \eqref{eq:case2} that
\begin{equation}
\label{eq:case2dec}
\bar{x}(t+1, \omega) - \bar{x}(t, \omega) \ge \frac{\alpha_t \gamma_t}{N}.
\end{equation}

\textit{Step 2c}, if for each agent $n$, $\abs{x_{n}(t, \omega) - \bar{\theta}_n(t, \omega)} > \gamma_t$, then from \eqref{eq:equal-sign2} we have
\begin{equation}
\label{eq:case3}
\begin{split}
    & \bar{x}(t+1, \omega) - \bar{x}(t, \omega) \\
    & = - \frac{\alpha_t \gamma_t}{N}\sum_{n \in [N]} \text{sign}\left[x_n(t, \omega) - \bar{\theta}_n(t, \omega) \right] \\
    & = - \frac{\alpha_t \gamma_t}{N} \sum_{n \in [N]} \text{sign}\left[ \bar{x}(t, \omega) - \theta_n \right]. 
\end{split}
\end{equation}
Since $\bar{x}(t, \omega)$ is less than the median(s), by counting number of $\theta_n$ on both sides of $\bar{x}(t, \omega)$ we still obtain \eqref{eq:case2dec}.

Apply similar argument as on (\ref{eq:even_avg}). There exists finite $T_4 \ge T_3$ such that $\bar{x}(t+1, \omega)$ is still less than the medians as $\bar{x}(t, \omega)$ is less than the medians. Take $T_5$ as the least $t \ge T_4$ such that $\bar{\gamma}_t - c_{\delta, \omega}(t+1)^{-\delta} - c_{\eta, \omega} (t+1)^{-\eta} \ge c_\gamma \gamma_t$ for some $0 < c_\gamma < 1$. Then in either \textit{Step 2a, 2b, 2c}, by \eqref{eq:distdec1}\eqref{eq:case2dec} we obtain 
\begin{equation*}
  \text{dist}(\bar{x}(t, \omega), \Theta) - \text{dist}(\bar{x}(t+1, \omega), \Theta) \ge  \frac{c_\gamma \alpha_t \gamma_t}{N}. 
\end{equation*}
Summing over all $t \ge T_5$ leads to a contradiction in that $\text{dist}(\bar{x}(T_5, \omega)) - \lim_{t \rightarrow \infty} \text{dist}(\bar{x}(t), \omega) = c_{\gamma}N^{-1} \sum_{t=T_6}^{\infty} \alpha_t \gamma_t = \infty$ by the choice $\tau_3 < 1 - \tau_1$, and thus establishes the desired assertion.
\end{proof}

\begin{proof}[Proof of Theorem 1]
    By Lemma \ref{lm:localcvg}, \ref{lm:globalcvg}, there exists some finite $T_1$ such that for all $t \ge T_1, \text{dist}(\bar{x}(t)) \le \bar{\gamma}(t)$ a.s. Thus, we have \[\mathbb{P}\left(\lim_{t \rightarrow \infty} (t+1)^{\tau_3} \text{dist}(\bar{x}(t), \Theta) = 0\right) = 1.\] By Lemma \ref{lm:aymcns} we have \[\mathbb{P}\left(\lim_{t \rightarrow \infty} (t + 1)^{\tau_1 - \tau_2 + \tau_3 - \epsilon_1} \|\mathbf{x}(t) - P_N \mathbf{x}(t)\|_2 = 0\right) = 1\] for every $0 < \epsilon_1 < \tau_1 - \tau_2$. For any $n \in [N]$, by the triangle inequality, we have 
	 \begin{equation*}
	 \begin{split}
	     \text{dist}(x_n(t), \Theta) &\le \abs{x_n(t) - \bar{x}(t)} + \text{dist}(\bar{x}(t), \Theta) \\
	     & \le \|\mathbf{x}(t) - P_N \mathbf{x}(t)\|_2 + \text{dist}(\bar{x}(t), \Theta).
	 \end{split}
	 \end{equation*} 
	 Then, it follows that for all $n$, we have \[\mathbb{P}\left(\lim_{t \rightarrow \infty} (t+1)^{\tau_3} \text{dist}(x_n(t), \Theta) = 0\right) = 1.\]
\end{proof}

\section{Numerical experiments}

We generate two random geometric graphs Graph \ref{gh1} and Graph \ref{gh2}. Both graphs consist of 40 nodes but have different connectivities measured by the second largest eigenvalue of the graph Laplacians. The Laplacians of Graph 1 and Graph 2 are $\lambda_2(L_1) \approx 1.8, \ \lambda_2(L_2) \approx 7.2$ respectively, and from Figs. \ref{gh1} and \ref{gh2} below it is clear that Graph \ref{gh2} is more densely connected. Additionally, both graphs are undirected and connected. We simulate random networks by assigning dropout probability for each link. In our experiments, we use dropout probabilities $0.1$ and $0.5$.

\begin{figure}[!htb]
   \begin{minipage}{0.22\textwidth}
     \centering
	\includegraphics[width=\linewidth]{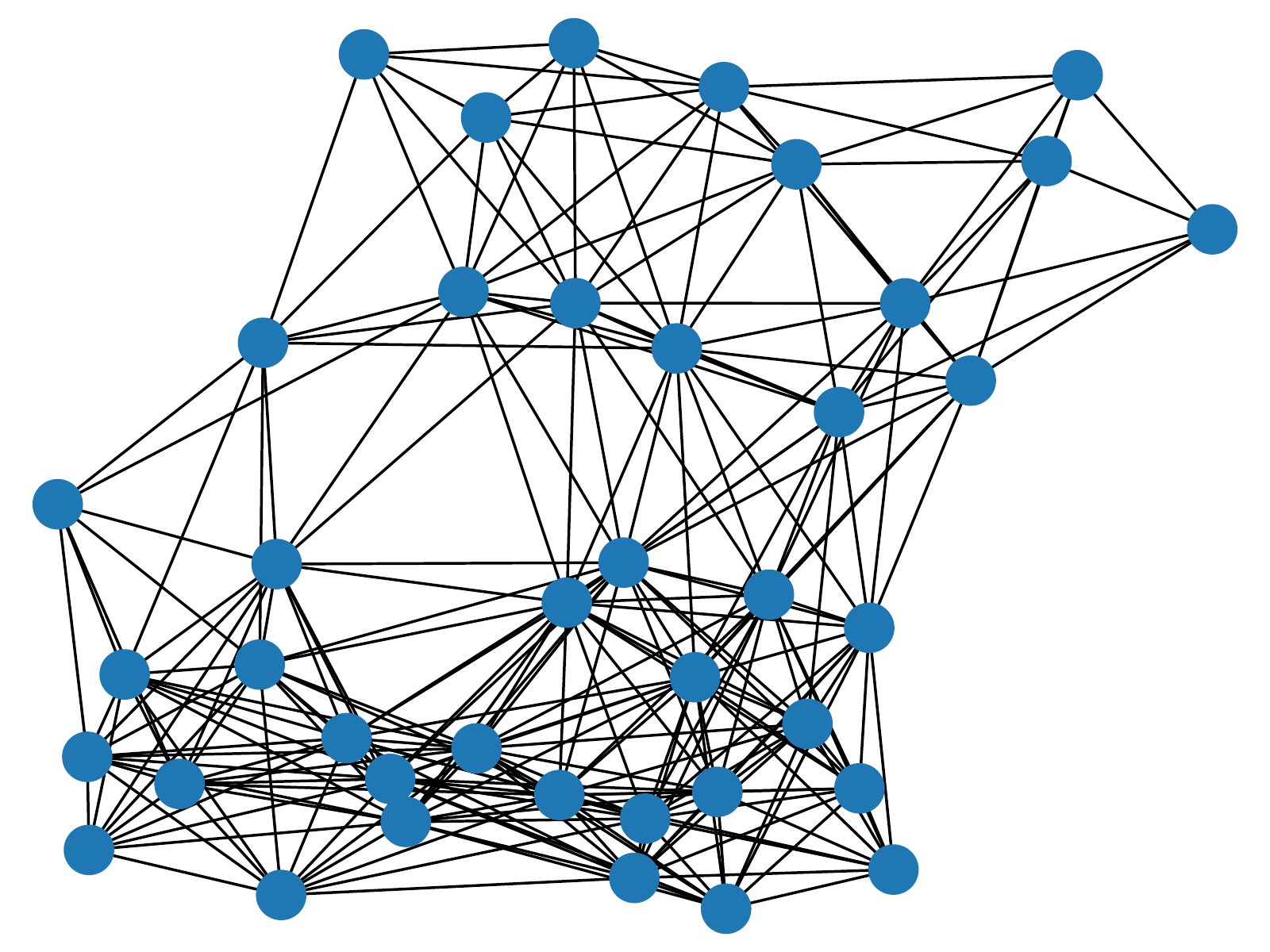}
	\caption{Graph 1, $\lambda_2(L_1) \approx 1.8$}
	\label{gh1}
   \end{minipage}\hfill
   \begin{minipage}{0.22\textwidth}
     \centering
	\includegraphics[width=\linewidth]{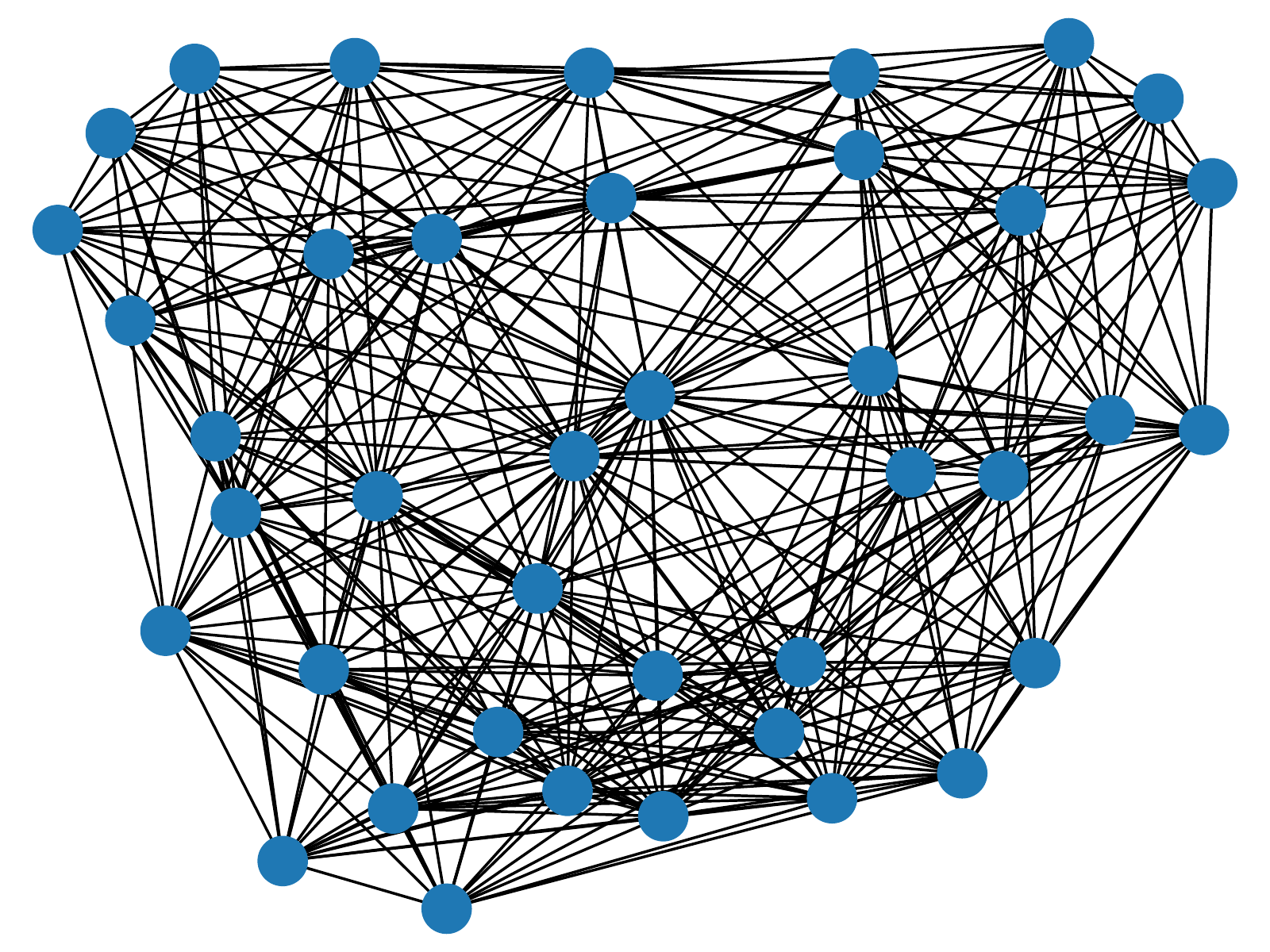}
	\caption{Graph 2, $\lambda_2(L_2) \approx 7.2$}
	\label{gh2}
   \end{minipage}
\end{figure}

We consider the problem setting $\theta_n = n$ for $n = 1, \ldots, 40$, $v_n(t) = 10/(t+1)$, and $w_n(t) \sim \mathcal{N}(0, 4)$. Note that we consider perturbations as the sum of a deterministic sequence and i.i.d. white noises. The deterministic sequence is not known to the agents as a prior. Since the largest possible deterministic errors as as tolerable by DMED are used, this problem is the hardest in the problem class we consider.
\begin{figure}[ht]
\centering
\includegraphics[width=0.48\textwidth]{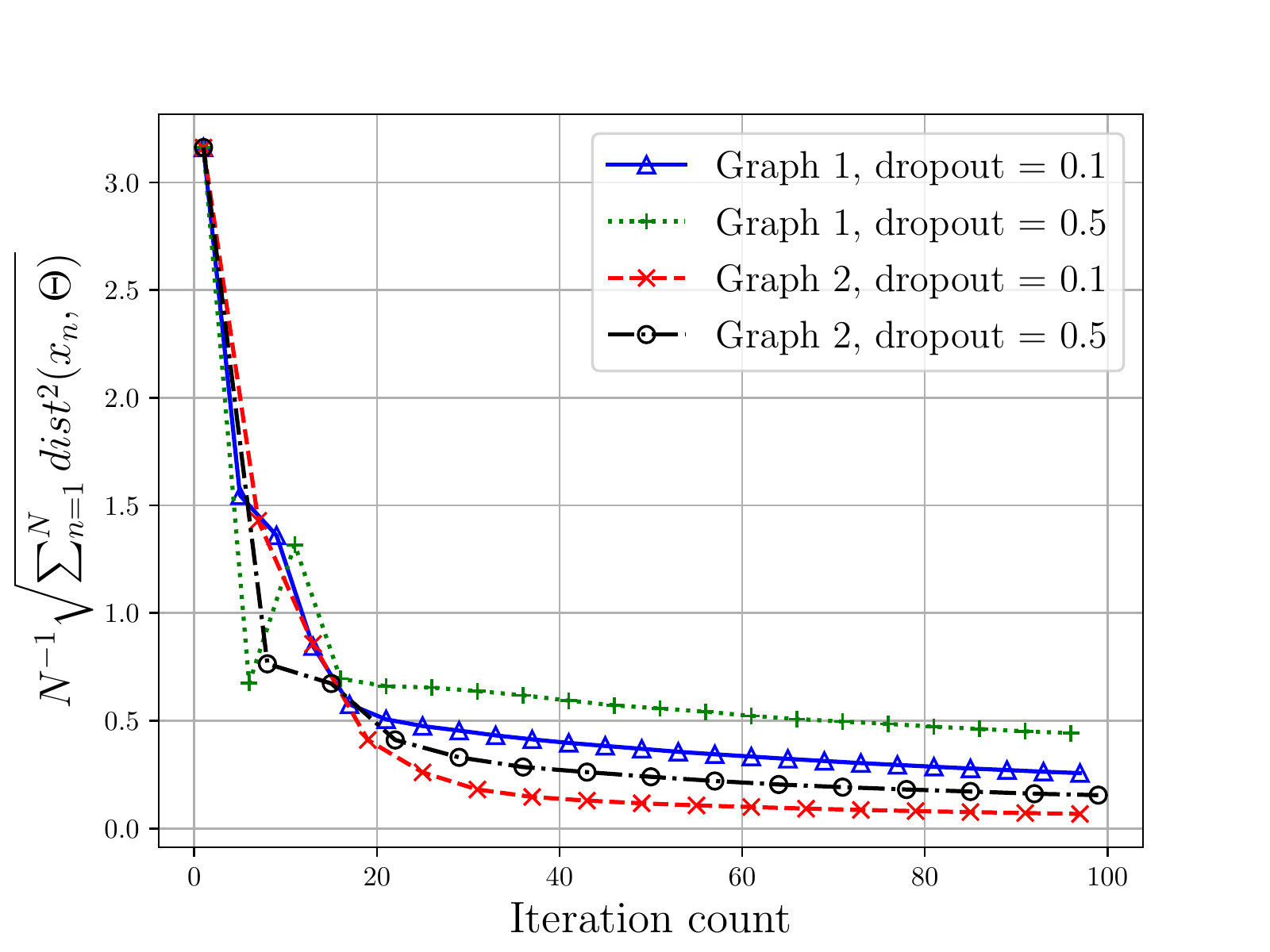}
\caption{Iteration count $t$ vs $N^{-1}\sqrt{\sum_{n=1}^N \text{dist}^2(x_n, \Theta)}$ with two networks with different dropout probability}
\label{fg:results}
\end{figure}

By assigning two different dropout probabilities $0.1$ and $0.5$ for each link in Graph \ref{gh1} and Graph \ref{gh2} respectively, we conduct experiments on $4$ random networks. In all $4$ random networks, we set the same parameters $\alpha_t = (t+1)^{-0.6}, \beta_t = (t+1)^{-0.2}/10, \gamma_t = 20(t+1)^{-0.3}, c_\mu = 10, \mu = 0.9$, and all local estimates start from 0. Given the random nature of the considered networks, we average the network average distance to a set of medians, i.e., $N^{-1}\sqrt{\sum_{n=1}^N \text{dist}^2(x_n, \Theta)}$ over $100$ network instances for each of $4$ experiments, and present the experiments results in Fig. \ref{fg:results}.

The simulation results in Fig. \ref{fg:results} demonstrate our theoretical findings that each local estimate converges to the set of medians sublinearly. The results validates the advantage of DMED over previous works that rely on connected networks, in that if a large proportion of links drop out at each time instance the DMED still converges to the set of medians sublinearly. From Fig. \ref{fg:results} we also observe that better connectivity and lower dropout probabilities could benefit the convergence speed of DMED. This emprical finding about convergence rate characteristics is not formally investigated in our analysis, while the intuition behind it is that better connectivity (measured by the second eigenvalue of the Laplacian) and lower dropout probabilities tend to speed up consensus type processes.

\section{Conclusion}
In this paper, we have studied the problem of dynamic median consensus over random networks that are required to connected only on average. We have considered a multi-agent networked setup in which each agent makes local observations, corrupted by decaying bias and white noise, of a \textit{distinct} value. The agents' objective is to estimate the median of the local values. We presented DMED, a \textit{consensus+innovations} type algorithm with \textit{clipped innovations} to address this problem. Under the DMED algorithm, the agents' local iterates converge almost surely to the set of medians at a sublinear rate. Finally, we validate the performance of our algorithm through numerical simulations. 

\bibliographystyle{ieeetr}
\bibliography{median}
\end{document}